\documentclass[12pt]{amsart}

\usepackage{color}

\usepackage{amsmath,amssymb,setspace,nicefrac, yhmath, amscd,eucal}
\usepackage[active]{srcltx}
\usepackage[colorlinks, linkcolor=red, citecolor=blue, urlcolor=blue, hypertexnames=true]{hyperref}
\usepackage{amsrefs}
\usepackage{txfonts}

\allowdisplaybreaks
\usepackage[all]{xy}

\newcommand{\supp}{\operatorname{supp}}
\newcommand{\Prob}{\operatorname{Prob}}

\newtheorem{thm}{Theorem}[section]
\newtheorem{cor}[thm]{Corollary}
\newtheorem{lem}[thm]{Lemma}

\newtheorem{definition}[thm]{Definition}

\newtheorem{remark}[thm]{Remark}
\newtheorem{proposition}[thm]{Proposition}

\theoremstyle{definition}

\begin{document}

\title[]{Singular subgroups in $\widetilde{A}_2$-groups and their von Neumann algebras}

\author{Yongle Jiang}\thanks{Y. Jiang was partially supported by the NCN (National Centre of Science) grant
2014/14/E/ST1/00525.}
\author{Piotr W. Nowak}\thanks{This project has received funding from the European Research Council (ERC) under the European Union's Horizon 2020 research and innovation programme (grant agreement no. 677120-INDEX)}
\address{Institute of Mathematics of the Polish Academy of Sciences, Warsaw, Poland}
\email{yjiang@impan.pl}
\email{pnowak@impan.pl}
\date{\today}
\maketitle

\begin{abstract}
We show that certain amenable subgroups inside $\widetilde{A}_2$-groups are singular in the sense of Boutonnet and Carderi. 
This gives a new family of examples  of singular group von Neumann subalgebras. 
We also give a geometric proof that if $G$ is an acylindrically hyperbolic group, $H$ is an infinite amenable subgroup containing 
a loxodromic element, then $H<G$ is singular. Finally, we present (counter)examples to show both situations happen concerning 
maximal amenability of $LH$ inside $LG$ if $H$ does not contain
loxodromic elements.     
\end{abstract}

\section{Introduction}

Let $M$ be a finite von Neumann algebra, $N$ be a von Neumann subalgebra of $M$ and denote by $\mathbb{E}_N$ the trace-preserving conditional expectation from $M$ onto $N$. A classical topic in von Neumann algebras is to study the relative position of $N$ inside $M$. There are two closely related notions to describe the relative position of $N$ inside $M$. One is singularity and the other one is maximal amenability.

Recall that $N$ is called \emph{singular} in $M$ \cite{Dix} if the normalizer of $N$, i.e. $\mathcal{N}(N):=\{u\in \mathcal{U}(M): uNu^*=N\}$, is contained in $N$. In general, it is not easy to decide whether given subalgebras, e.g. maximal abelian subalgebras (masas), are singular and this prompted Sinclair and Smith to introduce, a priori, stronger notion of singuality, which was called  strongly singularity in \cite{SS}. Recall that $N$ is said to be \emph{strongly singular}  if, for every unitary $u\in M$
\[\sup_{||x||\leq 1}||(\mathbb{E}_N-\mathbb{E}_{uNu^*})x||_2\geq ||(Id-\mathbb{E}_N)u||_2,\] 
where $||\cdot||_2$ denotes the $L^2$-norm associated with a prescribed faithful normal trace on $M$. Although the definition is more involved, it is easier to check, especially for group von Neumann subalgebras. For example, certain subgroups of hyperbolic groups are shown to give rise to strongly singular von Neumann subalgebras in \cite{SS}. Moreover, it was shown in \cite{SSWW}  that a singular masa is in fact also strongly singular for a separable II$_1$ factor $M$. 

Besides singularity, one also studies maximal amenability. Recall that $N$ is \emph{maximal amenable} in $M$ if $N$ is amenable and there are no amenable subalgebras in $M$ that strictly contain $N$.

Clearly, a maximal amenable von Neumann subalgebra is automatically singular. Although every nonamenable von Neumann algebra $M$ contains maximal amenable von Neumann subalgebras by Zorn's lemma, it is rather difficult to construct concrete examples of maximal amenable von Neumann subalgebras. 

The first such a concrete example is due to Popa. In \cite{popa_adv} Popa proved that the abelian von Neumann subalgebra generated by one of the generators of the non-abelian free group $F_n$, i.e. the generator masa, is maximal injective in the free group factor $L(F_n)$.
One ingredient in his proof is the so-called  ``asymptotic orthogonality property" for the generator masas inside $L(F_n)$. 
This method was later applied elsewhere, see e.g. \cite{remi_carderi1, hou}.

More recently, new techinques introduced in \cite{remi_carderi2} allowed to obtain more explicit examples of maximal amenable group von Neumann 
subalgebras that come from infinite maximal amenable subgroups.
This strategy is best suited for groups acting, in an appropriately regular way, on geometric objects and includes hyperbolic groups, many semisimple Lie groups of higher rank 
such as $\operatorname{SL}_3(\mathbb{Z})$.

In \cite{RS1996, RSS}  for groups acting on geometric objects, e.g. affine buildings, certain subgroups are shown to give rise to strongly singular von Neumann subalgebras. If we regard the homogeneous tree as a one dimensional affine building of type $\widetilde{A}_1$, then the degenerate case of the results in \cite{RS1996, RSS} states that the generator masas in $L(F_n)$ are strongly singular. Hence, it is natural to ask whether results in \cite{RS1996, RSS} can be strengthened to show the von Neumann subalgebras are actually maximal amenable.
The main result of our paper is an affirmative answer to this question, see Theorem \ref{mainthm}, Corollary \ref{cor: mainthm for A_2-groups}. The proof is based on the geometric approach in \cite{remi_carderi2} and previous work in \cite{RSS}.

We organize this paper as follows. In Section \ref{section: preliminaries}, we recall necessary background on affine buildings that are used in our main theorem. In Section \ref{section: on acylindrical hyperbolic groups}, we apply the geometric approach to acylindrically hyperbolic groups, which inspired questions studied in Section \ref{section: remove loxodromic elements assumption}. 

\section{Preliminaries and main theorem}\label{section: preliminaries}

\subsection{Affine buildings}
Let us briefly recall several standard facts on affine buildings, we refer the readers to \cite{brown, ronan} for more details.

Let $\Delta$ be an affine building. By $\Delta^0$ we denote its set of vertices.
Similarly, let $\mathcal{A}$ be an apartment, then $\mathcal{A}^0$ denotes its vertices.
We consider the boundary $\partial \Delta=X$, which is defined as the equivalence classes of sectors in $\Delta$. Recall that two sectors are \emph{equivalent} (or parallel) if their intersection contains a sector. Fix a special vertex $x\in \Delta$, for any $\omega\in\Omega$, there is a unique sector $[x, \omega)$ in the class $\omega$ having base vertex $x$ \cite[Theorem 9.6, Lemma 9.7]{ronan}. The \emph{boundary} $\Omega$ is defined to be the set of equivalent classes of sectors in $\Delta$. $\Omega$ also has the structure of a spherical building \cite[Theorem 9.6]{ronan} and topologically, $\Omega$ is a totally disconnected compact Hausdorff space and a basis for the topology is given by the set of the form $\Omega_x(v)=\{\omega\in\Omega: [x, \omega)~\mbox{contains}~v\}.$

Two boundary points $\omega, \bar{\omega}$ in $\Omega$ are said to be \emph{opposite} if the distance between them is the diameter of the spherical building $\Omega$. This is equivalent to the property that they are represented by opposite sectors $S, \bar{S}$ with the same base vertex in some apartment of $\Delta$ by \cite[Lemma 3.5]{JR}. 

For an $\omega$ in $\Omega$, we can define $\mathcal{O}(\omega)$ to be the set of all $\omega'\in\Omega$ such that $\omega'$ is opposite to $\omega$. Note that $\mathcal{O}(\omega)$ is an open set. Moreover, if $\omega\in \Omega$ and $\mathcal{A}$ is an apartment in $\Delta$, then there exists a boundary point $\bar{\omega}$ of $\mathcal{A}$ such that $\bar{\omega}$ is opposite $\omega$ \cite[Lemma 3.2]{JR}. As a corollary, if $\omega_1,\ldots, \omega_n$ are the boundary points of an apartment, then $\Omega=\mathcal{O}(\omega_1)\cup\cdots\cup \mathcal{O}(\omega_n)$ \cite[Corollary 3.3]{JR}.

Motivated by \cite[Section 5]{RSS} we fix a group $G$ of 
automorphisms of an affine building $\Delta$ with boundary $\Omega$ satisfying the following properties.
\begin{itemize}
\item[(B1)] $G$ acts freely on the vertex set $\Delta^0$ with finitely many vertex orbits (i.e. cocompactly).
\item[(B2')]  There is an apartment $\mathcal{A}$ in $\Delta$ and an amenable subgroup $H$ of $G$ such that $H$ preserves $\mathcal{A}$ and $H\setminus \mathcal{A}$ is compact, i.e. $\mathcal{A}$ is a \emph{periodic} apartment. In particular, $H\setminus \mathcal{A}^0$ is finite, where $\mathcal{A}^0$ is the vertex set of $\mathcal{A}$.
\item[(B3)] The natural mapping $H\setminus \mathcal{A}^0\to G\setminus \Delta^0$ is injective.
\end{itemize}

\begin{remark}\normalfont
(1) In \cite[Section 5]{RSS}, an almost identical set of conditions is introduced. The only difference is condition (B2) therein, which was stated as follows (using our notations): There is an apartment $\mathcal{A}$ in $\Delta$ and an abelian subgroup $H$ of $G$ such that $H\setminus \mathcal{A}^0$ is finite, where $\mathcal{A}^0$ is the vertex set of $\mathcal{A}$. It was made clear in \cite[Remark 5.1(b)]{RSS} that the sole reason for assuming that $H$ is abelian in condition (B2) is to obtain an abelian von Neumann algebra $LH$. Everything else works equally well without this assumption. Here, we use notation (B2') to distinguish it from condition (B2). Note that both (B2) and (B2') are checked and applied for the same type of examples.\\
(2) As observed in \cite[Lemma 5.2]{RSS}, since $G$ acts freely on $\Delta^0$, condition (B3) guarantees that for any $g\in G$, $g\mathcal{A}^0\cap \mathcal{A}^0\neq\emptyset$ implies $g\in H$.\\
(3) The above notion of periodic apartments was called ``doubly periodic apartments" in \cite{RR, RS1996}; while in \cite{JR, RSS}, it was simply called periodic apartments.  
\end{remark}

Let $\mathcal{A}$ be the periodic apartment appeared in condition (B2'), and fix a special vertex $z$ in $\mathcal{A}$. As explained in \cite[P. 207]{JR}, we choose a pair of opposite sectors $W^+, W^-$ in $\mathcal{A}$ based at $z$ and denote by $\omega^{\pm}$ the boundary points represented by $W^{\pm}$, respectively. By periodicity of $\mathcal{A}$, there is a periodic direction represented by a line $L$ in the sector direction of $W^+$. This means that there exists some element $u\in G$ which leaves $L$ invariant and translates the apartment $\mathcal{A}$ in the direction of $L$. Then $u^n\omega^+=\omega^+, u^n\omega^-=\omega^-$ for all $n\in \mathbb{Z}$.  

One ingredient for our proof is \cite[Proposition 3.7]{JR}, which shows that $\omega^-$ is an attracting fixed point for $u^{-1}$. We state it below (using our notations) for readers' convenience.
 
\begin{proposition}[Proposition 3.7 in \cite{JR}]\label{prop: prop 3.7 in JR's paper}
Let $G$ acts properly and cocompactly on an affine building $\Delta$ with boundary $\Omega$. Let $\mathcal{A}$ be a periodic apartment and choose a pair of opposite boundary points $\omega^{\pm}$. Let $u\in G$ be an element which translates the apartment $\mathcal{A}$ in the direction of $\omega^+$. Then $u^{-1}$ attracts $\mathcal{O}(\omega^+)$ towards $\omega^-$; that is, for each compact subset $K$ of $\mathcal{O}(\omega^+)$ we have $\lim_{n\to\infty}u^{-n}(K)=\{\omega^-\}$.
\end{proposition}
We note that during the proof of this proposition, the authors introduced an increasing family of compact open sets $K_0\subset K_1\subset K_2\subset \cdots$ such that $\cup_{N=0}^{\infty}K_N=\mathcal{O}(\omega^+)$ and they actually proved that $\lim_{n\to\infty}u^{-n}(K_N)=\{\omega^-\}$ for each $N\geq 0$.

\subsection{$\widetilde{A}_2$-groups}
An $\widetilde{A}_2$-group acts simply transitively on the vertices of an affine building of type $\widetilde{A}_2$. Such groups were studied in \cite{CMSZ, CMSZ2} through a combinatorial description, i.e. the so-called \emph{triangle presentation}. 
We would not recall the definition here, but refer the readers to \cite[Introduction]{RS1996} for a clean presentation and the above papers for details. 

$\widetilde{A}_2$-groups have Kazhdan's property (T) by \cite[Theorem 4.6]{CMS} and operator algebras associated with $\widetilde{A}_2$-groups were studied extensively, see e.g. \cite{RR, RS1996, JR, RSS, R2006}.

\subsection{Singular subgroups}
Let us recall the notion of singular subgroups as introduced in \cite{remi_carderi2}.

\begin{definition}\label{definition: singular subgroup}
Consider an amenable subgroup $H$ of a discrete countable group $G$. Suppose that $G$ acts continuously on the compact space $X$. 
We say that $H$ is singular in $G$ (with respect to $X$) if for any $H$-invariant probability measure $\mu$ on $X$ 
and $g \in G\setminus H$ we have  $g\cdot \mu \perp \mu$.
\end{definition}

For convenience we will denote by $\Prob_H(X)$ the space of $H$-invariant probability measures on $X$.

It turns out that with the presence of singularity, an amenable subgroup is automatically maximal amenable \cite[Lemma 2.2]{remi_carderi2}. More importantly, this fact is also witnessed at the level of von Neumann algebras as shown in the following theorem.

\begin{thm}[Theorem 2.4 in \cite{remi_carderi2}]
Suppose $G$ is a discrete countable group admitting an amenable, singular subgroup $H$. Then for any trace preserving action $G\curvearrowright (Q, \tau)$ on a finite amenable von Neumann algebra, $Q\rtimes H$ is maximal amenable inside $Q\rtimes G$.
\end{thm}

\subsection{Main theorem and its proof}
Now, we are ready to state our main theorem, which is a strengthening of \cite[Theorem 5.8]{RSS}.

\begin{thm}\label{mainthm}
Let $G$ be a group of automorphisms of a locally finite affine building $\Delta$ with boundary $\Omega$. Assume that (B1), (B2'), (B3) hold and $H\subseteq G$ 
is as described in condition (B2'). Then $H$ is a singular subgroup in $G$.
\end{thm}

\begin{proof}
According to Definition \ref{definition: singular subgroup}, we need to show that for any $\mu\in \Prob_{H}(\Omega)$ and every $g\in G\setminus H$, we have $g\cdot \mu\perp \mu$.
Consider such a $\mu \in \Prob_H(\Omega)$ and denote by $\{w_1,\ldots, w_k\}$ the boundary points of the apartment $\mathcal{A}$ appeared in condition (B2'). Note that $\Omega$ is a spherical building and $k$ equals the cardinality of the spherical Weyl group, which is finite.

We claim that $\supp(\mu)\subseteq \{w_1,\ldots, w_k\}$. Indeed, assume the contrary and take any 
$w\in \supp(\mu)\setminus \{w_1,\ldots,w_k\}$. Since $w\in  \supp(\mu)$, 
we may take a small closed  neighborhood of $w$, say $N_w$, such that $N_w\cap \{w_1,\ldots, w_k\}=\emptyset$ and $\mu(N_w)>0$. Since the boundary points of the apartment $\mathcal{A}$ are exactly $w_1,\ldots, w_k$, we may apply \cite[Corollary 3.3]{JR} to deduce $\Omega=\mathcal{O}(w_1)\cup\ldots\cup\mathcal{O}(w_k)$, where $\mathcal{O}(w)$ is the set of all $w'\in \Omega$ such that $w'$ is opposite to $w$. 

Without loss of generality, we may assume $\mu(N_w\cap \mathcal{O}(w_1))>0$ and that $w_k$ is the opposite boundary point of $w_1$. $N_w\cap \mathcal{O}(w_1)$ may not be a compact subset of $\mathcal{O}(w_1)$, but one may replace it with the intersection with some $K_n$ defined in the proof of \cite[Proposition 3.7]{JR}. See the paragraph after Proposition \ref{prop: prop 3.7 in JR's paper} for a quick explanation.

Hence we obtain a compact subset with $\mu(N_w\cap \mathcal{O}(w_1)\cap K_n)>0$. Note that \cite[Proposition 3.7]{JR} applies since condition (B1) guarantees that the action of $G$ on the vertex set $\Delta^0$ is proper and cocompact and $\mathcal{A}$ is periodic by condition (B2').

Without loss of generality, we assume $N_w\cap \mathcal{O}(w_1)$ is a compact subset of $\mathcal{O}(w_1)$. Then by \cite[Proposition 3.7]{RSS}, we know that $\lim_{n\to\infty}u^{-n}(N_w\cap \mathcal{O}(w_1))=\{w_k\}$, where $u\in G$ is an element which translates the apartment $\mathcal{A}$ in the direction of $w_1$. 

Note that $u$ satisfies the extra property $u\omega_1=\omega_1$, this implies that $u\in H$. Indeed, since we have sector representatives for $u\omega_1$ and $\omega_1$ in the apartment $u\mathcal{A}$ and $\mathcal{A}$ respectively, $u\mathcal
{A}^0\cap \mathcal{A}^0$ contains a subsector; in particular, $u\mathcal{A}^0\cap \mathcal{A}^0\neq \emptyset$. Then by \cite[Lemma 5.2]{JR}, condition (B3) implies $u\in H$.

Since all $w_i$ are fixed points under $H$, we deduce for any $n\geq 1$, $w_k\not\in u^{-n}(N_w)$, which implies $w_k\not\in \overline{u^{-n}(N_w\cap \mathcal{O}(w_1))}$, 
since $N_w$ is closed in $\Omega$. Therefore, we may find an increasing sequence $n_i\to\infty$ such that $u^{-n_i}(N_w\cap \mathcal{O}(w_1))\cap u^{-n_j}(N_w\cap \mathcal{O}(w_1))=\emptyset$ for all $i\neq j$. Hence, we deduce that $1=\mu(\Omega)\geq \mu(\sqcup_{i=1}^{\infty}u^{-n_i}(N_w\cap \mathcal{O}(w_1)))=\sum_{i=1}^{\infty}\mu(u^{-n_i}(N_w\cap \mathcal{O}(w_1)))=\sum_{i=1}^{\infty}\mu(N_w\cap\mathcal{O}(w_1))=\infty$, a contradiction.

We now claim that $g\cdot \supp(\mu)\cap \supp(\mu)=\emptyset$ for any $g\in G\setminus H$. To see this, assume the contrary. Then $gw_i=w_j$ for some $i, j\in \{1,\ldots, k\}$ by the above claim. Since we have sector representatives for $gw_i$ and $w_j$ in the apartment $g\mathcal{A}$ and $\mathcal{A}$ repectively, $g\mathcal{A}^0\cap \mathcal{A}^0$ contains a subsector; in particular, $g\mathcal{A}^0\cap\mathcal{A}^0\neq \emptyset$. Then by \cite[Lemma 5.2]{RSS}, condition (B3) implies $g\in H$, a contradiction.

Then, combining the above two claims, we deduce that $g\cdot \mu\perp\mu$ for all $g\not\in H$.
\end{proof}
Applying Theorem \ref{mainthm} to $\widetilde{A}_2$-buildings, we have the following corollary. 
\begin{cor}\label{cor: mainthm for A_2-groups}
Let $G$ be an $\widetilde{A}_2$ group acting on an $\widetilde{A}_2$-building $\Delta$ and $H<G$ be an abelian subgroup which acts simply transitively on the vertex set of an apartment $\mathcal{A}$ in $\Delta$. Then $H$ is singular in $G$.
\end{cor}
The above result is a strengthening of \cite[Theorem 2.8]{RS1996}.

Indeed, in the above example, $H\cong \mathbb{Z}^2$ by \cite[P. 6]{RS1996} and the apartment $\mathcal{A}$ is (doubly) periodic \cite[p. 6--7]{RS1996}. By \cite[Example 5.9]{RSS}, we know all conditions (B1), (B2'), (B3) are satisfied. 

As explained in \cite[Example 5.9]{RSS} or \cite[Remark 1.5]{RS1996}, we can apply the above corollary to $G$ being the groups (4.1), (5.1), (6.1), (9.2), (13.1) and (28.1) in the table of the end of \cite{CMSZ2}. 

Note that $\widetilde{A}_2$ groups have Kazhdan's property (T) by \cite[Theorem 4.6]{CMS} and they give rise to II$_1$ factors by \cite[Lemma 0.2]{RS1996} or \cite[Lemma 5.6]{RSS}. So we have more examples of higher rank abelian, maximal amenable subalgebras in II$_1$ factors with property (T). See \cite{remi_carderi1, remi_carderi2} for more examples.

\section{Acylindrically hyperbolic groups}\label{section: on acylindrical hyperbolic groups}

In \cite[p. 1201]{remi_carderi1}, it was mentioned that if $H<G$ is an infinite amenable subgroup which is hyperbolically embedded then $LH$ is maximal amenable inside $LG$. Since the proof was based on Popa's asymptotic orthogonality approach and was omitted, we take this opportunity to include a proof of a slightly weaker version (see Remark \ref{rk: hyperbolic embeded vs containing loxodromic}) of this result using the geometric approach in \cite{remi_carderi2}. The proof is similar to the proof of \cite[Lemma 3.2]{remi_carderi2}, but uses more recent work on acylindrically hyperbolic groups. 

Let us first briefly recall the standard terminology related to acylindrically hyperbolic groups, we refer the readers to \cite{DGO, Osin} for details.

An action of a group $G$ on a metric space $S$ is called \emph{acylindrical} if for every $\epsilon>0$ there exist $R, N>0$ such that for every two points $x, y$ with $d(x, y)>R$, there are at most $N$ elements $g\in G$ satisfying $d(x, gx)\leq \epsilon$ and $d(y, gy)\leq \epsilon$. From now on, we assume the space $S$ is hyperbolic and $G$ acts on $S$ isometrically, this action extends to an action on its Gromov boundary $X:=\partial S$ by homeomorphisms. We say an element $g\in G$ is \emph{loxodromic} if the map $\mathbb{Z}\to S$ defined by $n\mapsto g^ns$ is a quasi-isometry for some (equivalently, any) $s\in S$. Every loxodromic element $g\in G$ has exactly two limit points $g^{\pm \infty}$ on $\partial S$. Loxodromic elements $g, h\in G$ are called \emph{independent} if the sets $\{g^{\pm \infty}\}$ and $\{h^{\pm\infty}\}$ are disjoint.

We say the action $G\curvearrowright S$ is \emph{elementary} if the limit set of $G$ on $\partial S$ contains at most two points. Here, the limit set of $G$ is just the set of accumulation points of a $G$-oribts on $\partial S$. In fact, this definition does not depend on the choice of $G$-orbits.

$G$ is called an \emph{acylindrically hyperbolic group} if it admits a non-elementary acylindrical action on a hyperbolic space $S$. Typical examples of acylindrically hyperbolic groups include non-elementary hyperbolic groups, certain non-virtually-cyclic relatively hyperbolic groups, mapping class groups and $Out(F_n)$ for $n\geq 2$ etc.

A useful tool used later is the following theorem of Osin on classification of groups acting acylindrically on hyperbolic spaces. Note that for an acylindrically hyperbolic group $G$ (w.r.t. $G\curvearrowright S$), condition (3) below holds.

\begin{thm}[Theorem 1.1 in \cite{Osin}]
Let $G$ be a group acting acylindrically on a hyperbolic space $S$ (isometrically). Then $G$ satisfies exactly one of the following three conditions.
\begin{enumerate}
\item $G$ has bounded orbits.
\item $G$ is virtually cyclic and contains a loxodromic element.
\item $G$ contains infinitely many independent loxodromic elements.
\end{enumerate}
\end{thm}

We are now in the position to state the following result.
\begin{thm}\label{thm on acylindrically hyperbolic groups case}
Let $G$ be an acylindrically hyperbolic group (say w.r.t the action $G\curvearrowright S$) and $H$ be any maximal amenable subgroup containing a loxodromic element $h$ (w.r.t. $G\curvearrowright S$). Then $LH<LG$ is maximal injective.
\end{thm}
This is a direct corollary of the following proposition which proves that $H$ is singular in $G$.
\begin{proposition}\label{thm on MCG}
Let $H<G$ be an infinite maximal amenable subgroup containing a loxodromic element $h$. Let $X=\partial S$. Then the following statements hold.
\begin{enumerate}
\item There exist two points $a, b\in X$ such that $H$ is the stabilizer of the set $\{a, b\}$, that is $H=\operatorname{Stab}_G(\{a, b\}):=\{g\in G: g\cdot\{a, b\}=\{a, b\}\}$.
\item Any $H$-invariant probability measure on $X$ is of the form $t\delta_a+(1-t)\delta_b$ for some $t\in [0, 1]$.
\item Any element $g\in G\setminus H$ is such that $g\cdot\{a, b\}\cap\{a, b\}=\emptyset$.
\end{enumerate}
\end{proposition}

For the proof, we record the following lemma.

\begin{lem}\label{same stabilizer groups for loxodromic elements}
Let $a, b$ be the two fixed points of the loxodromic element $h$ in $X=\partial S$. Then $\operatorname{Stab}_G(a)=\operatorname{Stab}_G(b)$.
\end{lem}
\begin{proof}
Assume not, then either $\operatorname{Stab}_G(a)\setminus \operatorname{Stab}_G(b)\neq \emptyset$ or $\operatorname{Stab}_G(b)\setminus \operatorname{Stab}_G(a)\neq \emptyset$. If $g\in \operatorname{Stab}_G(a)\setminus \operatorname{Stab}_G(b)$. Then $ghg^{-1}$ is also loxodromic by definition. And note that $Fix(h)=\{a, b\}$, but $b\not\in Fix(ghg^{-1})\ni a$. Hence, for each $t\in \langle ghg^{-1}, h\rangle$, $a\in Fix(t)$. Then by \cite[Theorem 1.1]{Osin}, $\langle ghg^{-1}, h\rangle$ is virtually cyclic and contains a loxodromic element $t$. Then $e\neq gh^ng^{-1}=h^{n'}\in \langle t\rangle$ for some nonzero integers $n, n'$. Then $Fix(ghg^{-1})=Fix(gh^ng^{-1})=Fix(h^{n'})=Fix(h)$, a contradiction. The other case is proved similarly.
\end{proof}

\begin{proof}[Proof of Proposition \ref{thm on MCG}]
By \cite[Proposition 3.4]{Hamann}, $h$ acts on $X$ with a north-south dynamics. Denote by $a, b$ the two fixed points of $h$ in $X=\partial S$, and let us assume $a$ is the attracting point.

(i) Let $s\in H$. Then $shs^{-1}$ is a loxodromic element with fixed points $s\cdot a$ and $s\cdot b$. If $\{a, b\}\cap \{s\cdot a, s\cdot b\}=\emptyset$, then by the ping-pong lemma, $H\supseteq \langle h, shs^{-1} \rangle$ contains a free group, which is impossible since $H$ is amenable. Then by Lemma \ref{same stabilizer groups for loxodromic elements}, $\{s\cdot a, s\cdot b\}=\{a, b\}$ since $shs^{-1}$ and $h$ fix a common point and hence the other point. Hence $H\subseteq \operatorname{Stab}_{G}(\{a, b\})$. To show that equality holds we note that $\operatorname{Stab}_{G}(\{a, b\})$ is amenable since $[\operatorname{Stab}_G(\{a, b\}): \operatorname{Stab}_G(a)\cap \operatorname{Stab}_G(b)]\leq 2$ and $\operatorname{Stab}_G(a)$ is virtually cyclic by \cite[Theorem 1.1]{Osin}.

(ii) We only need to show the support of any $H$-invariant probability measure is contained in $\{a, b\}$. This is a consequence of the north-south dynamics action of $h$. We sketch the proof for completeness. Assume there exists $p\in \supp(\mu)\setminus\{a, b\}$, then since $X$ is complete Hausdorff (i.e. for any two distinct points $u, v\in X$, there are open sets $U, V$ containing $u, v$ respectively, such that $\bar{U}\cap \bar{V}=\emptyset$, see \cite{Sun}), we may find a closed neighborhood $\mathcal{O}_p$ of $p$ such that $\mathcal{O}_p\cap \{a, b\}=\emptyset$ and $\mu(\mathcal{O}_p)>0$. Then there exists an increasing sequence $n_i$ such that $h^{n_i}\mathcal{O}_p\to a$ and the family of sets $\{h^{n_i}\mathcal{O}_p\}_i$ is pairwise disjoint, hence we get a contradiction since $\mu(h^{n_i}\mathcal{O}_p)=\mu(\mathcal{O}_p)$.

(iii) By Lemma \ref{same stabilizer groups for loxodromic elements} and (1), we know that $\operatorname{Stab}_G(a)=\operatorname{Stab}_G(b)\subseteq H$. Then the proof goes similarly as in \cite{remi_carderi2}. We include it for completeness. Take $g\in G$ such that $g\cdot a=b$. If there exists some $s\in H$ which exchanges $a$ and $b$. Then $sg$ fixes $a$ and so $g\in H$. Otherwise, all elements in $H$ fix $a$ and $b$, then $gsg^{-1}$ fixes $b$ and $g^{-1}sg$ fixes $a$, for all $s\in H$. Hence $g$ normalizes $H$ so $g\in H$ by maximal amenability.
\end{proof}

\begin{remark}
$H$ always exists since every group element $g\in G$ is contained in a maximal amenable subgroup by Zorn's lemma. And such an $H$ is virtually cyclic by \cite[Theorem 1.1]{Osin}.
\end{remark}

\begin{remark}\label{rk: hyperbolic embeded vs containing loxodromic}
\normalfont
Note that in Theorem \ref{thm on acylindrically hyperbolic groups case}, the subgroup $H$ is hyperbolically embedded by \cite[Theorem 6.8]{DGO}. It is not clear to us whether one can find a loxodromic element inside any infinite hyperbolically embedded amenable subgroup $H$. It is known that if $H$ does not contain any loxodromic elements, then $H$ is elliptic by \cite[Theorem 1.1]{Osin}. On the one hand, there do exist elliptic subgroups that are not hyperbolically embeded, see \cite[Corollary 7.8]{MO};
on the other hand, if a hyperbolic embeded subgroup $H$ is virtually cyclic, then it contains a loxodromic element by the proof of $(L_4)\Rightarrow (L_1)$ in the proof of \cite[Theorem 1.4]{Osin}. 
\end{remark}

\section{The case of no loxodromic elements in $H$}\label{section: remove loxodromic elements assumption}

Motivated by Theorem \ref{thm on acylindrically hyperbolic groups case}, it is natural to ask whether we can drop the assumption that $H$ contains loxodromic elements or more generally $H$ is hyperbolically embedded. Moreover, by \cite{osin_betti number}, we know that many non-amenable groups with positive first $\ell^2$-Betti number are acylindrically hyperbolic, then it is natural to ask whether $LH$ is maximal amenable in $LG$ if $H<G$ is infinite maximal amenable and $\beta^{(2)}_1(G)>0$.  Modifying the example in \cite[p. 1697]{remi_carderi2}, we show that both questions have negative answers.

\begin{proposition}\label{counterexample}
Let $K=BS(m, n)=\langle a,t|ta^mt^{-1}=a^n \rangle$, $H=\langle a \rangle<K$ and $G=K*F_2$, where $F_2$ denotes the non-abelian free group on two generators. Then the following statements hold.
\begin{enumerate}
\item $\beta^{(2)}_1(G)>0$, and if $|m|, |n|\geq 3$, then $H$ is maximal amenable in $G$ but $LH$ is not maximal amenable in $LG$. 
\item $K$ is not acylindrically hyperbolic if $|m|, |n|\geq 3$, while $G$ is acylindrically hyperbolic and $H<G$ is not hyperbolically embeded. 
\item If $c: G\to \mathcal{H}$ is any cocycle with $c(a)=0$, then $c(t)=0$, i.e. $ker(c)\neq H$, where $G\curvearrowright \mathcal{H}$ is a mixing unitary representation.
\end{enumerate}
\end{proposition}
\begin{proof}
(1) By \cite[Proposition 3.1]{ja},  $\beta^{(2)}_1(G)\geq \beta^{(2)}_1(K)+\beta^{(2)}_1(F_2)\geq 1$. As explained in \cite[p. 1697]{remi_carderi2}, $H$ is maximal amenable in $K$ if $|m|, |n|\geq 3$ (see Proposition \ref{prop: max amenable in BS(n, n)} below for a different proof) but $LH$ is not maximal amenable $LK$ since $x:=\sum_{k=0}^{n-1}a^ktat^{-1}a^{-k}\in \mathbb{C}K$ commutes with $LH$, hence $LH$ is not maximal amenable in $LG$ either. Then we can apply Proposition \ref{prop on maximal amenable in free product} below to see $H$ is still maximal amenable in $G$ since $K$ is torsion free if $mn\neq 0$ by \cite{kms}. 
 
(2) If $|m|, |n|\geq 3$, then $K$ is not acylindrically hyperbolic by \cite[Example 7.4]{Osin}. While $G$ is acylindrically hyperbolic by \cite{MO} or \cite[Corollary 1.3 or Theorem 1.1]{osin_betti number}, and observe that $H<G$ is not almost malnormal since $tHt^{-1}\cap H$ is infinite, hence it is not hyperbolically embedded by \cite[Lemma 7.1]{Osin} or \cite[Proposition 2.8]{DGO}.

(3) From $c(ta^nt^{-1})=c(a^m)=0$, we deduce that $c(t)=ta^nt^{-1}c(t)$. Then since $ta^nt^{-1}$ has infinite order, we get $||c(t)||^2=\langle c(t), (ta^nt^{-1})^kc(t) \rangle\to 0$ as $k\to\infty$.
\end{proof}

\begin{proposition}\label{prop on maximal amenable in free product}
Let $H, K$ and $L$ be countable discrete groups. If $H$ is maximal amenable in $K$, and both $K$ and $L$ are torison free. Then $H$ is also maximal amenable in $K*L$. 
\end{proposition}
\begin{proof}
First, we observe that it suffices to show $K$ is free from $g$ for every $g\in K*L\setminus K$, i.e. $\langle K, g \rangle=K*\langle g \rangle$. 

To see this, one just check that for all $g\in K*L\setminus H$, $\langle H, g\rangle$ is not amenable. If $g\in K*L\setminus K$, then $\langle H, g \rangle=H*\langle g \rangle\geq F_2$ by assumption. If $g\in K\setminus H$, then $\langle H, g \rangle\subseteq K$ is not amenable since $H<K$ is maximal amenable.

We are left to show for all $g\in K*L\setminus K$, $K$ is free from $g$. 

Claim 1: for every $e\neq k\in K$ and every $g\in K*L\setminus K$, $K$ is free from $g$ if and only if $K$ is free from $kg$.

\begin{proof}[Proof of Claim 1]
By symmetry, it suffices to show $\Rightarrow$ holds.

Suppose $k_1(kg)^{m_1}\cdots k_i(kg)^{m_i}=e$ for some $k_2,\dots  k_i\in K\setminus{\{e\}}, k_1\in K$, $m_1\cdots m_{i-1}\neq 0$ and $m_i\in \mathbb{Z}$. Then, since $K$ is free from $g$, we deduce $|m_1|, \ldots, |m_{i-1}|=1$; otherwise, by looking at the middle word pieces between any two successive $g^{\pm}$, we deduce $k=e$, a contradiction.

Then, we divide the argument into four cases.

Case 1: $m_j=1$ for all $j\in \{1, \ldots, i-1\}$. By freeness, we deduce $k_jk=e$ for all $j\in \{1,\ldots, i-1\}$ and hence $g^{i-1}k_i(kg)^{m_i}=e$. If $m_i=0$ or $m_i=-1$, then $k_i=e$, a contradiction. If $m_i=1$, then $k_ik=e$ and $g^i=e$, this is a contradiction since $K*L$ is torsion free. If $|m_i|\geq 2$, then $k=e$, a contradiction.

Case 2: $m_j=-1$ for all $j\in \{1, \ldots, i-1\}$. The proof is similar to the proof of case 1. 

Case 3: $m_1=1$ and there exists the smallest $j\in \{1,\dots, i-1\}$ such that $m_j=-1$. Then by freeness, we must have $k_j=e$, a contradiction.

Case 4: $m_1=-1$ and there exists  the smallest $j\in \{1,\dots, i-1\}$ such that $m_j=1$. Then by freeness, we must have $k^{-1}k_jk=e$, i.e. $k_j=e$ a contradiction.
\end{proof}  
By Claim 1 and taking inverses it is also clear that for any $e\neq k\in K$, $K$ is free from $g$ if and only if $K$ is free from $gk$. Hence, to prove $g$ is free from $K$, we may assume when written in reduced form, either $g=x$ or $g=xty$, where $x, y\in L\setminus{\{e\}}$ and $e\neq t$ is a reduced word in $G$ with head and tail come from $K$. Then clearly, $g$ is free from $K$.
\end{proof}

\begin{proposition}\label{prop: max amenable in BS(n, n)}
Let $G=BS(m, n)=\langle a, t~|~ ta^mt^{-1}=a^n\rangle$ and $H=\langle a \rangle$. Then $H$ is maximal amenable in $G$ if $|m|, |n|\geq 3$.
\end{proposition}

\begin{proof}

It suffices to prove $K:=\langle g, a \rangle$ contains free group $F_2$ for every $g\in G\setminus H$.

By the normal form theorem for HNN extension \cite[Page 182]{ls}, for every $e\neq g\in G$, we may write $g$ in reduced normal form, i.e. $g=a^{i_0}t^{\epsilon_1}a^{i_1}t^{\epsilon_2}\cdots t^{\epsilon_k}a^{i_k}$, where $\epsilon_i\in \{\pm 1\}$ and no substrings of the form $ta^{m*}t^{-1}$ or $t^{-1}a^{n*}t$ appear, where $m*$ (respectively, $n*$) denotes any integer divisible by $m$ (respectively, $n$). Moreover, if $\epsilon_j=1$ for some $1\leq j\leq k$, then $0\leq i_{j}<m$; similarly, if $\epsilon_j=-1$ for some $1\leq j\leq k$, then $0\leq i_j<n$.

Notice that $K=\langle a^{-i_0}ga^{-i_k}, a \rangle$ and $a^{-i_0}ga^{-i_k}\in G\setminus H$, so without loss of generality, we may assume that $g=t^{\epsilon_1}a^{i_1}t^{\epsilon_2}\cdots t^{\epsilon_k}$ in reduced normal form.

Then, using Britton's lemma (see \cite{britton} or \cite[Page 181]{ls}), one can check that $gag^{-1}a$ is free from $agag^{-1}$ and both have infinite order if $|m|, |n|\geq 3$; in other words, $F_2\cong \langle gag^{-1}a, agag^{-1} \rangle\subseteq K$. 
\end{proof}

Despite the existence of the above examples, we also have examples showing that some maximal amenable but not hyperbolically embeded subgroups may give rise to maximal amenable group von Neumann algebras. Indeed, let $G=(\mathbb{Z}\times F_2)*F_2=(\langle a \rangle\times \langle b, c \rangle)*F_2$, $K=\mathbb{Z}\times F_2=\langle a \rangle\times \langle b, c \rangle$ and $H=\mathbb{Z}^2=\langle a, b \rangle$. Since $\langle a\rangle\subseteq cHc^{-1}\cap H$ is infinite, $H$ is not almost malnormal; therefore it is not hyperbolically embeded in the acylindrically hyperbolic group $G$. While $LH$ is maximal amenable in $LK$ by \cite[Theorem 2.4]{remi_carderi2}, hence $LH$ is still maximal amenable in $LG=LK*LF_2$ since any amenable subalgebra (in $LG$) containing $LH$ is contained in $LK$ by \cite{hou} or \cite{ozawa}.

\end{document}